\documentclass[12pt]{extarticle}
\usepackage{amsmath, amsthm, amssymb, mathtools, hyperref, xcolor}
\usepackage[english]{babel}
\usepackage{csquotes}
\usepackage[shortlabels]{enumitem}
\usepackage{graphicx}
\usepackage{adjustbox}
\usepackage{xspace}
\usepackage{tikz}
\usetikzlibrary{arrows,calc,fit,matrix,positioning,cd}
\usepackage{cleveref}
\usepackage{comment}
\usepackage{todonotes}

\oddsidemargin \evensidemargin
\definecolor{cb-yellow}{RGB}{221,170,51}
\definecolor{cb-red} {RGB}{187,85,102}
\definecolor{cb-green}{RGB}{17,119,51}

\hypersetup{
    pdfauthor={Thomas Kahle, Hal Schenck, Bernd Sturmfels, Maximilian Wiesmann},
    pdftitle={The Likelihood Correspondence},
    pdfsubject={},
    pdfkeywords={},
    colorlinks=true,
    linkcolor=cb-red,
    filecolor=cb-red,      
    urlcolor=cb-green,
    citecolor=cb-green,
}

\newtheorem{theorem}{Theorem}
\numberwithin{theorem}{section}
\newtheorem{proposition}[theorem]{Proposition}
\newtheorem{lemma}[theorem]{Lemma}
\newtheorem{corollary}[theorem]{Corollary}

\theoremstyle{definition}
\newtheorem{definition}[theorem]{Definition}
\newtheorem{remark}[theorem]{Remark}
\newtheorem{example}[theorem]{Example}

\newcommand{\ZZ}{\mathbb{Z}}
\newcommand{\NN}{\mathbb{N}}
\newcommand{\PP}{\mathbb{P}}
\newcommand{\CC}{\mathbb{C}}

\newcommand{\cA}{\mathcal{A}}
\newcommand{\cE}{\mathcal{E}}
\newcommand{\cF}{\mathcal{F}}

\newcommand{\cL}{\mathcal{L}}
\newcommand{\cO}{\mathcal{O}}
\newcommand{\R}{\mathcal{R}}
\newcommand{\V}{V} 
\newcommand{\Qi}{Q_{\backslash 1}}

\DeclareMathOperator{\coker}{coker}
\DeclareMathOperator{\Hom}{Hom}

\DeclareMathOperator{\Sym}{Sym}
\DeclareMathOperator{\pd}{pdim}
\DeclareMathOperator{\rk}{rk}
\DeclareMathOperator{\codim}{codim}

\DeclareMathOperator{\Der}{Der}
\DeclareMathOperator{\ann}{ann}

\newcommand{\midcolon}{\, : \,}
\newcommand\scalemath[2]{\scalebox{#1}{\mbox{\ensuremath{\displaystyle #2}}}}

\date{}

\title{\textbf{The Likelihood Correspondence}}
\author{Thomas Kahle, Hal Schenck, Bernd Sturmfels, and Maximilian Wiesmann
}

\begin{document}
\maketitle

\begin{abstract}    \noindent       
    An arrangement of hypersurfaces in projective space is 
    strict normal crossing (SNC) 
    if and only if its Euler discriminant is nonzero.  We study the critical loci of 
    arbitrary Laurent monomials in the equations
    of the smooth hypersurfaces. The family of these loci forms an
     irreducible variety in the product of two projective spaces,
    known in algebraic statistics as the likelihood correspondence 
    and in particle physics as the scattering correspondence.
    We establish an explicit determinantal representation for the minimal generators of
        the bihomogeneous prime ideal that defines this variety.
\end{abstract}

\section{Introduction}
\label{sec:intro}
We study complements of projective hypersurface arrangements. When the hypersurfaces are linear, these are hyperplane arrangements, which  have a long tradition in combinatorics,  geometry, topology, and singularity theory.
Commutative algebra is also a key player, due to the importance of the ideals and modules 
one attaches to arrangements.  A central object is the module of derivations that are tangent to the hyperplanes.
Terao's longstanding conjecture states that freeness of this module
is determined by the underlying matroid~\cite{OT}.

In spite of their ubiquity in the mathematical literature, 
hyperplanes and their matroids are too restrictive
for many applications. The arrangements 
encountered {\em in the sciences} are usually non-linear \cite{ArrangementsAndLikelihood}. More often than not one faces
hypersurfaces of degree two or~more.

We fix an arrangement 
in projective space $\PP^{n-1}$ defined by $m$
homogeneous  polynomials $f_1,f_2,\ldots,f_m$ of
degrees $d_1,d_2,\ldots,d_m$ in  $ \CC[x_1,\ldots,x_n]$.
Let $\mathcal{A} = \V(f)$ for $f = f_1 f_2 \cdots f_m$. We
consider functions on $\PP^{n-1} \backslash \mathcal{A}$
that are given by Laurent monomials in these polynomials:
\begin{equation}
\label{eq:likelihood}
 f^s \,\, = \,\, f_1^{s_1} f_2^{s_2} \cdots f_m^{s_m}.  \end{equation}
 Here, $s_1,s_2,\ldots,s_m$ are integers such that the expression (\ref{eq:likelihood})
is  homogeneous of degree zero:
\begin{equation}
\label{eq:momentumconservation}
d_1 s_1 \,+ \, d_2 s_2 \,+\, \cdots \,+\, d_m s_m \,\, = \,\, 0. 
\end{equation}
 After choosing an appropriate branch of the complex logarithm, we pass to
 the function
\begin{equation}
\label{eq:loglikelihood}
\ell_\mathcal{A} \,\, = \,\,
{\rm log}(f_1^{s_1} f_2^{s_2} \cdots f_m^{s_m})\,\,=\,\,
s_1  {\rm log}(f_1) \,+\,
s_2  {\rm log}(f_2) \,+\,
\cdots + \,
s_m  {\rm log}(f_m) .
 \end{equation}
The product (\ref{eq:likelihood}) is the
 {\em master function} in singularity theory~\cite{cohen2011critical}. Its logarithm
  (\ref{eq:loglikelihood})  is the {\em scattering potential} in physics~\cite{Lam24}.
  Using terms from statistics~\cite{BCF23},
 we refer to  (\ref{eq:likelihood})  as the {\em likelihood function}
and  (\ref{eq:loglikelihood}) as~the {\em log-likelihood function}.
The latter is defined for arbitrary complex numbers $s_1,s_2,\ldots,s_m$.
   We assume (\ref{eq:momentumconservation}) so that
 (\ref{eq:loglikelihood}) makes sense on projective space~$\PP^{n-1}$.

The likelihood correspondence $\mathcal{L}_{\mathcal{A}}$ 
of the arrangement $\mathcal{A}$ is a subvariety in  $\PP^{m-1} \times \PP^{n-1}$.
It is the Zariski closure of the set of pairs $(s,x)$,
where $s = (s_1,s_2,\dotsc,s_m) \in \CC^{m}$, and 
$x = (x_1,x_2,\dotsc,x_n) \in \CC^{n} \backslash \mathcal{A}\,$
is any (non-branching) critical point of the function~$\ell_\mathcal{A}$.
See \cite[\S 1]{HS14} and~\cite{ArrangementsAndLikelihood} for
background and the version for implicitly defined arrangements.
\begin{definition}
The {\em likelihood correspondence} $\cL_{\cA}$ is the Zariski closure in $\PP^{m-1}\times \PP^{n-1}$ of
\begin{equation}
\label{eq:likcor}
  \left\lbrace (s, x)\in \CC^{m}\times \CC^{n}\,\colon\,
    \frac{\partial \ell_\cA}{\partial x_i}(x,s)=0,\, i=1,\dots,n,\,
    f^s(x) \neq 0,\, F(x)\in X_{\text{reg}} \right\rbrace.
\end{equation}
Here $X$ denotes the Zariski closure of the image of
$F\colon\CC^n\rightarrow\CC^{m},\, x \mapsto (f_1(x),\dots ,
f_m(x))$, and $X_{\text{reg}}$ is its set of nonsingular points.  The
\emph{likelihood ideal}~$I_\cA$ is the vanishing ideal of~$\cL_\cA$.
\end{definition}

Our Theorems \ref{thm:main} and \ref{thm:main2} 
furnish an explicit matrix description of the likelihood ideal~$I_\mathcal{A}$,
under the assumption that the arrangement $\cA$ is {\em generic}.
The following definition makes our notion of genericity precise.
See \cite[\href{https://stacks.math.columbia.edu/tag/0CBN}{Lemma~41.21.2}]{stacks} 
and the illustration in Example \ref{ex:1plus3}.
\begin{definition}
The arrangement $\cA$ is {\em strict normal crossing (SNC)} if each $f_i$ is irreducible 
and for any subset $I\subset \{1,\dotsc, m\}$ the variety $\bigcap_{i\in I} V(f_i)$
is smooth of codimension $|I|$.
\end{definition}

The SNC hypothesis is equivalently characterized by the non-vanishing 
of the {\em Euler discriminant}, which we introduce in Section \ref{sec:main_result}.
In Theorem~\ref{thm:Euler_discriminant_equals_mixed_discriminant} we show
that the SNC hypothesis corresponds to the Euler characteristic
of $\PP^{n-1} \backslash \cA$ attaining its generic value.

The computational advantage of our new matrix description is dramatic.
For a test case, consider
 $m=5$ quadrics in $\PP^3$, so $n=4$. Clearing
dominators of the rational functions~(\ref{eq:likcor}),
 and then saturating in {\tt Macaulay2},
has no chance of success. The advances
in \cite[Section~5]{ArrangementsAndLikelihood} help a little bit, but not much.
Our Theorem \ref{thm:main} simply gives a formula for the answer: 
the ideal $I_\cA$ has $29$ minimal generators: the minors of 
the  $6 \times 8$ matrix in (\ref{eq:ourmatrix}) together 
with \eqref{eq:momentumconservation}.

Our results are of independent interest in commutative algebra: Theorem~\ref{thm:main} implies that the likelihood module $M(\cA)$  
from \cite{ArrangementsAndLikelihood}
is of \emph{linear type},
when $\cA$ is an SNC divisor.  This means that the Rees algebra of  $M(\cA)$ is equal to the symmetric algebra. 
Equivalently,  $I_\cA$ equals the \emph{pre-likelihood ideal}~\cite{ArrangementsAndLikelihood}. We conclude 
that SNC arrangements are {\em gentle} (Corollary~\ref{cor:genericImpliesGentle}).

Catanese et al.~\cite[Theorem 1]{catanese2006maximum} showed that 
the {\em ML degree}, i.e.~the number of critical points of $\ell_\cA$, is always
bounded above by the coefficient of $z^{n-1}$ in the
generating~function 
\begin{equation*}
  \frac{(1-z)^n}{(1-d_1z)(1-d_2z)\cdots(1-d_mz)}~.
\end{equation*}
Equality holds when the arrangement $\cA$ is SNC. More generally,
we show in Lemma \ref{lem:generic_multidegree} that the
coefficients of this generating function encode the multidegree of the likelihood ideal $I_\cA$.
    
\smallskip

We begin with an arrangement of ML degree $13$. This
  illustrates the concepts above.
  
\begin{example}[A line and three conics] \label{ex:1plus3}
Let $n=3,m=4$, $d_1 = 1$, $d_2=d_3=d_4=2$.
The line $f_1$ and the conics
$f_2,f_3,f_4$ 
form an SNC arrangement $\mathcal{A}$ in the
projective plane $\PP^2$ when
\begin{enumerate}
\item each of the three conics is irreducible; \vspace{-0.2cm}
\item no pair among the four curves is tangent; \vspace{-0.2cm}
\item no triple among the four curves intersects.
\end{enumerate}
The likelihood
correspondence $\mathcal{L}_{\mathcal{A}}$ is an irreducible
surface in $\PP^3 \times \PP^2$.
We shall compute its
 prime ideal $I_\cA$ in 
$ \CC[s,x] = \CC[s_1,s_2,s_3,s_4,x_1,x_2,x_3]$,
assuming the SNC conditions above.

The surface $\mathcal{L}_{\mathcal{A}}$ is defined by 
the three partial derivatives of the log-likelihood function:
\begin{equation}
\label{eq:direct}
\frac{s_1}{f_1} \frac{\partial f_1}{\partial x_i} \, + \,
\frac{s_2}{f_2} \frac{\partial f_2}{\partial x_i} \, + \,
\frac{s_3}{f_3} \frac{\partial f_3}{\partial x_i} \, + \,
\frac{s_4}{f_4} \frac{\partial f_4}{\partial x_i} \,\,\,
= \,\,\, 0 \qquad {\rm for} \,\, i=1,2,3. 
\end{equation}
If we multiply these equations by $x_1,x_2,x_3$, and add them up, then Euler's relation reveals
\begin{equation}
\label{eq:euler}
 s_1 \, + \, 2 s_2 \, + \, 2 s_3 \, + \, 2 s_4\,\, = \,\, 0 . 
\end{equation}
To solve (\ref{eq:direct}),
we now replace the third equation
$\partial \ell_\cA/\partial x_3 = 0$ with
the linear equation (\ref{eq:euler}).
We are left with two rational functions of degree $-1$. 
By clearing denominators, we obtain two polynomials
of degree $(1,6)$ in $\CC[s,x]$, i.e.~they are linear in $s$ and of degree $6$ in~$x$.

Given a general point $s $ on the plane (\ref{eq:euler}) in $\PP^3$,
we intersect two plane curves of degree~$6$. This yields
$36$ points in $\PP^2$, in accordance with
B\'ezout's Theorem. However,  
$23$ of these~$36$ are extraneous. This includes 
the $2\!+\!2\!+\!2\!+\!4\!+\!4\!+\!4 = 18$ pairwise
intersection points $\{f_i = f_j = 0\}$. We also have
five extraneous solutions with
$x_3=0$, since using (\ref{eq:euler}) in our derivation
is equivalent to using
 $x_3 \frac{\partial \ell_\cA}{\partial x_3} = 0$.
 Therefore, $\ell_\cA$ has $13$ critical points
in $\PP^2 \backslash \mathcal{A}$.

The likelihood ideal $I_\cA$ lives in the polynomial ring
$\CC[s,x]$ with $7$ variables. It can be computed from the numerators in
 (\ref{eq:direct}) by saturating with respect to the product $f = f_1 f_2 f_3 f_4 $. 

 Our matrix description of $I_{\cA}$ circumvents this saturation step and is far more efficient.
It shows that $I_\cA$ has six minimal generators:
the one in (\ref{eq:euler})
of degree $(1,0)$, three of degree $(1,4)$, and two of degree $(1,5)$.
The last five are the
maximal minors of the $4 \times 5$ matrix
\begin{equation}
\label{eq:4by5} \begin{pmatrix}
\,f_2 & 0 & 0 &  \frac{\partial f_2}{\partial x_1} &  \frac{\partial f_2}{\partial x_2} &  \frac{\partial f_2}{\partial x_3} 
\,\medskip \\
\, 0  & f_3 & 0 & \frac{\partial f_3}{\partial x_1} &  \frac{\partial f_3}{\partial x_2} &  \frac{\partial f_3}{\partial x_3}
\, \medskip \\
\, 0  &  0 & f_4 &  \frac{\partial f_4}{\partial x_1} &  \frac{\partial f_4}{\partial x_2} &  \frac{\partial f_4}{\partial x_3}
\,  \smallskip \\
\,s_2 & s_3 & s_4 & 0 & 0 & 0 \\
\end{pmatrix}
\scalemath{0.89}{\begin{pmatrix} 
\,\,  1 & \,0 &\, 0 & 0 & 0 \smallskip \\
\,\,  0 & \,1 &\, 0 & 0 & 0 \smallskip \\
\,\, 0 & \,0 &\, 1 & 0 & 0 \smallskip \\
\,\, 0 & \,0 &\, 0 & \frac{\partial f_1}{\partial x_3} & 0 \smallskip \\
\,\, 0 & \,0 &\, 0 & 0 &\frac{\partial f_1}{\partial x_3} \smallskip \\
\,\, 0 & \,0 &\, 0 & \!\!-\frac{\partial f_1}{\partial x_1} &\! \!\! -\frac{\partial f_1}{\partial x_2}\, \\
\end{pmatrix}.}
\end{equation}
This matrix product description of our ideal $I_\cA$ will be explained after Theorem \ref{thm:main2}.
\end{example}

This paper is organized as follows.
In Section \ref{sec:main_result} we present our main results, 
starting with all relevant definitions, and examples to illustrate the key concepts. The proofs are spread over the two sections thereafter.
Section \ref{sec:det_ideals} develops tools from commutative algebra.
The focus is on determinantal ideals, circuit syzygies, and the Buchsbaum--Rim resolution.

In Section \ref{sec:multidegrees} we determine
the multidegree of the generic likelihood correspondence,
and we show that it matches that of the proposed determinantal ideal using the Giambelli--Thom--Porteous Formula.
Theorem \ref{thm:Euler_discriminant_equals_mixed_discriminant}
unifies algebraic and topological approaches to
the Euler discriminant.
Finally, in Section \ref{sec:homological} we connect the likelihood correspondence to the module $\Der(\cA)$ of derivations tangent to $\mathcal{A}$, as well as to certain Rees and symmetric algebras.

We close the introduction with a remark about the significance
of genericity. Admittedly, the scenarios used by practitioners
are almost never generic. Models from physics or statistics (cf.~\cite[Section 3]{ArrangementsAndLikelihood})
do not satisfy the SNC hypothesis. However, for many applied projects,
the generic case can be the best point of departure.
In our setting, for any special model~$\cA$,  the matrix (\ref{eq:ourmatrix}) drops rank
on the likelihood correspondence $\mathcal{L}_\cA$.
We found that matrix to be a much better input for computing $I_\cA$
 than the rational functions in (\ref{eq:likcor}).
 Many success stories in applied algebraic geometry
begin with concise theorems about generic instances.
\section{Main Results}\vskip -.06in
\label{sec:main_result}

Let $f_1,f_2,\ldots,f_m $ be 
homogeneous elements of degrees $d_1,d_2,\ldots,d_m $ in
the polynomial ring $R = \CC[x_1,\ldots,x_n]$.
These define an arrangement $\mathcal{A}$ of $m$ algebraic
hypersurfaces in $\PP^{n-1}$. With each hypersurface we associate a parameter~$s_i$. 
The coordinate ring of $\PP^{m-1} \times \PP^{n-1}$ 
is denoted $S = R[s_1,\ldots,s_m]$.  It is equipped with a $\ZZ^2$-grading under which $\deg(s_i) = (1,0)$
for all $i$ and $\deg(x_j) = (0,1)$ for all~$j$.
Our object of study is the prime ideal $I_\mathcal{A} \subset S$ that defines
the likelihood correspondence $\mathcal{L}_\mathcal{A}$.
This ideal is homogeneous in the $\ZZ^2$-grading on~$S$.

We assume throughout that $m \geq n+1$, and that
the given arrangement $\cA$ is SNC.
%
If a $\ZZ^2$-homogeneous polynomial $f$  in $S$ has bidegree $(1,d)$,
for some $d \in \NN$, then we simply say that $f$ {\em has degree $d$}.
We shall see that all generators of the likelihood ideal $I_\mathcal{A}$
have this property. In particular, the linear relation
  (\ref{eq:momentumconservation}) is a generator of degree $0$.
  All other generators of $I_\mathcal{A}$ have positive degree, and we list them according to their degrees.

To this end, we now introduce some polynomial matrices.
First, there is the Jacobian matrix 
 $J_f = (\partial f_i/\partial x_j)$ of size $m \times n$.
 Using this, we introduce the $m \times (m+n)$ matrix
\vskip -.15in
\begin{equation*}
 Q \,\,\, = \,\,\, \bigl(\,\,{\rm diag}(f_1,f_2,\ldots,f_m) \,\,|\, \,J_f \,\,\bigr).
\end{equation*}
The matrix $Q^s$ is obtained by augmenting $Q$ with
the row vector $(s_1,s_2,\ldots,s_m, 0,0,\ldots,0)$.
We write $Q_{\backslash i}$ for the $m \times (m+n-1)$ matrix obtained from
$Q$ by deleting the $i$th column. We similarly define the matrix $Q^s_{\backslash i}$
with $m+1$ rows and $m+n-1$ columns. In particular,
\begin{equation}
\label{eq:ourmatrix}
    Q^s_{\backslash 1} \,\,= \,\,\begin{pmatrix}
	\, 0 & \cdots & 0 & \frac{\partial f_1}{\partial x_1} & \frac{\partial f_1}{\partial x_2} & \cdots & \frac{\partial f_1}{\partial x_n\ } \smallskip
\\
 \,        f_2 & \cdots & 0 & \frac{\partial f_2}{\partial x_1} & \frac{\partial f_2}{\partial x_2} & \cdots & \frac{\partial f_2}{\partial x_n\ } \smallskip \\
   \,	 \vdots & \ddots & \vdots & \vdots & \vdots & \ddots & \vdots \smallskip \\
\,         0 & \cdots & f_m & \frac{\partial f_m}{\partial x_1} & \frac{\partial f_m}{\partial x_2} & \cdots & \frac{\partial f_m}{\partial x_n\ }  \smallskip \\
\,  s_2 & \cdots & s_m & 0 & 0 & \cdots & 0 
    \end{pmatrix}.
\end{equation}

\begin{theorem} \label{thm:main} Suppose that $d_1,d_2,\ldots,d_m \geq 2$
and the arrangement $\mathcal{A}$ is SNC.
Then the likelihood ideal 
$I_\mathcal{A}$ is minimally generated by  (\ref{eq:momentumconservation})
and the
$\binom{m+n-1}{n-2}$ maximal minors of the matrix~$ Q^s_{\backslash 1} $.
For $\,i \in \{2,3,\ldots,n\}$, we have 
$\, \binom{n}{i} \cdot \binom{m-1}{i-2} \,$ generators in degree
$d_1 + d_2 + \cdots + d_m - i$. These generators are the maximal
minors that use
 precisely $i$ of the last $n$ columns in~$ Q^s_{\backslash 1} $.
\end{theorem}

This theorem provides a concise determinantal representation for
the likelihood correspondence $\mathcal{L}_\mathcal{A}$ of any 
model in algebraic statistics parametrized by generic polynomials. For $1$-dimensional models
 $(n=2)$, the matrix $ Q^s_{\backslash 1} $ is square, and 
 $I_\mathcal{A} = \bigl\langle \,\sum_{i=1}^m d_i s_i\,,\, \det(\Qi^s)  \bigr\rangle$.
 In general, $I_\mathcal{A}$ has codimension $n$.  
 We shall see that the likelihood ideal $I_\mathcal{A}$ is Cohen--Macaulay.

\begin{remark} It is instructive to examine what happens when $m$ drops below $n+1$.
If $m \leq n$ then each $s_i$ is a factor of some maximal minor of $ Q^s_{\backslash 1} $, and
 our ideal  is not prime. 
One of the minimal primes is $\langle s_1,s_2,\ldots,s_m \rangle$.
However, our determinantal ideal is still radical and Cohen--Macaulay, and
the count of minimal generators in Theorem \ref{thm:main} 
remains valid.
\end{remark}

We next state a more general theorem that also applies
 when some of the hypersurfaces
are hyperplanes. Suppose that precisely $L$ of the given polynomials
$f_1,f_2,\ldots,f_m$ have degree~$1$. We set $\ell = {\rm min}(n,L)$.
We have $\ell=0$ in Theorem \ref{thm:main},
and we have $\ell=1$ in Example \ref{ex:1plus3}.

\begin{theorem} \label{thm:main2}
If $\mathcal{A}$ is SNC then the
likelihood ideal  $I_\mathcal{A}$ is minimally generated by (\ref{eq:momentumconservation})
and $\binom{m+n-1-\ell}{n-2}$ of the maximal minors of $\, Q^s_{\backslash 1} $.
We have
$ \binom{n-\ell}{i-\ell} \cdot \binom{m-1}{i-2} $ generators in~degree
$ d_1 + d_2 + \cdots + d_m-i$ for
${\rm max}(2,\ell) \leq i \leq n$.
These arise from the columns
as in Theorem~\ref{thm:main}.
\end{theorem}

In Theorem \ref{thm:main2}, the matrix $Q^s_{\backslash 1}$ has constant
entries $\partial f_i / \partial x_j$ in the $\ell$ rows
where ${\rm degree}(f_i) = 1$. Using such entries for
column operations, we  replace the 
$(m+1) \times (m+n-1)$ matrix $Q^s_{\backslash 1}$ with a
matrix of size $(m+1-\ell) \times (m+n-1-\ell)$
whose maximal minors are precisely the minimal generators of $I_\mathcal{A}$.
This was shown in (\ref{eq:4by5}) for $n=3,m=4,\ell=1$.

\begin{corollary}\label{cor:mainLinear}
If $\mathcal{A}$ is an arrangement in $\PP^{n-1}$ that includes
$n$ or more hyperplanes, then its likelihood ideal $I_\mathcal{A}$  has
$\binom{m-1}{n-2}$ minimal generators, all of  degree $d_1+d_2+\cdots+d_m-n$.
In particular, if $\mathcal{A}$ is a hyperplane arrangement then
the $\binom{m-1}{n-2}$ generators have degree $m-n$.
\end{corollary}

Theorem \ref{thm:main2}  can be derived
from Theorem \ref{thm:main} using elementary row and
column operations. We now explain this
for the special case in Corollary \ref{cor:mainLinear}.
After a linear change of basis, we may
assume that $f_{m-n+1},\ldots,f_m$
are the coordinate functions $x_1,\ldots,x_n$.  This gives rise to an identity block in the Jacobian.  By column operations 
and by deleting rows, we  transform  the  matrix $Q^s_{\backslash 1}$ into
the following matrix, which has only $m-n+1$ rows and $m-1$ columns,
and whose $\binom{m-1}{m-n+1}$ maximal minors
are the minimal generators of $I_\cA$: 
\vspace{-1.6ex}
\[ 
    \scalemath{0.98}{
      \!\!  \begin{pmatrix}
	 0 & \cdots & 0 &  0 &  \cdots & \!\!0 & \frac{\partial f_1}{\partial x_1} & \frac{\partial f_1}{\partial x_2} & \cdots & \! \frac{\partial f_1}{\partial x_n} \smallskip \\
        f_2 & \cdots & 0 &  0 &  \cdots & \!\!0 & \frac{\partial f_2}{\partial x_1} & \frac{\partial f_2}{\partial x_2} & \cdots & \! \frac{\partial f_2}{\partial x_n} \smallskip \\
   	 \vdots & \ddots & \vdots & \vdots &  
   \vdots & \vdots &\!\! \vdots &  \vdots &\ddots & \vdots 
   \smallskip \\
         0 & \cdots &\! f_{m-n} \!&  0 &  \cdots & \!\! 0 & \! \frac{\partial f_{m-n}}{\partial x_1}\! &
\! \frac{\partial f_{m-n}}{\partial x_2} \!& \cdots &\!  \frac{\partial f_{m-n}}{\partial x_n} \! \smallskip \\
  s_2 & \cdots &\!s_{m-n} \!& \! s_{m-n+1} \! & \cdots &\!\! s_m & 0 & 0 & \cdots & 0 
    \end{pmatrix} \!
    }
    \scalemath{0.84}{
    \begin{pmatrix}
    1 & 0 &  \cdots  &             0 &  0 & \cdots & 0 \\
    0 & 1 &  \cdots  &              0 &  0 & \cdots & 0 \\ 
    \vdots & \vdots & \ddots &  0 &0 & \cdots & 0 \\
    0 &  0 &  \cdots & 1 &         0 & \cdots & 0 \\
    0  & 0 & \cdots & 0 &           1 & \cdots & 0 \\
    0  & 0 & \cdots & 0 &           0 & \ddots & 0 \\
0  & 0 & \cdots & 0 &           0 & \cdots & 1 \smallskip \\
    0 & 0 & \cdots &\!\!\! - x_1 &          0 & \cdots & 0 \\
     0 & 0 & \cdots&  0 &       \!  \! \!-x_2    & \cdots & 0 \\
\vdots & \vdots & \ddots & \vdots  & \vdots & \ddots & 0       \\
     0 & 0 & \cdots  & 0 &               0    & \cdots & \!\!\!\!-x_n  \\
\end{pmatrix}  
      }  .
\]

We next introduce the Euler discriminant. This is the polynomial
which furnishes the algebraic representation of
the SNC hypothesis,
which is essential
for Theorems \ref{thm:main} and~\ref{thm:main2}.
Every arrangement $\mathcal{A}$, given by polynomials $f_1,f_2,\ldots,f_m$, is
 a point in the parameter space
 \begin{equation}
 \label{eq:parameterspace}
 \PP \,\, := \,\,
\mathbb{P}^{\binom{d_1+n-1}{n-1}-1} \times
 \mathbb{P}^{\binom{d_2+n-1}{n-1}-1} \times \cdots \times
\mathbb{P}^{\binom{d_m+n-1}{n-1}-1} .
\end{equation}
Let $I$ be any non-empty subset of cardinality $\leq n$ in 
$[m] = \{1,2,\ldots,m\}$. The discriminant $\Delta_I$ is defined as
the unique (up to scaling) irreducible polynomial  on $\PP$ which vanishes
whenever $\,\bigl\{x \in \PP^{n-1} :f_i(x) = 0 \,\,{\rm for} \,\, i \in I \bigr\}$ is not a smooth variety
of codimension $| I |$. 

The polynomial $\Delta_I$ was named the {\em $\mathcal{A}_I$-discriminant}
by Gel'fand, Kapranov and Zelevinvsky \cite{gelfand2008discriminants}. 
Here $\mathcal{A}_I$  denotes the set of lattice points in a
polytope of dimension $n-2+|I|$, namely the {\em Cayley polytope}
of the scaled simplices $d_i \Delta_{n-1}$ for $i \in I$. 
 The degree of the $\mathcal{A}_I$-discriminant is given in
 Theorem 2.8 in Section 9.2.D of  \cite{gelfand2008discriminants}.
Namely, it equals $ \sum_F\, (-1)^{{\rm codim} \,F} ({\rm dim}  \,F + 1) \cdot {\rm vol}(F)$, where the sum is over all faces $F$ of the Cayley polytope ${\rm conv}(\cA_I)$.
This formula relies on the fact that the corresponding
toric variety is smooth.

We define the {\em Euler discriminant} to be the product
of the discriminants $\Delta_I$ where $I$ runs over all
non-empty subsets of cardinality at most $n$ in $[m]$.
The factor $\Delta_{\{i\}}$~is the usual {\em discriminant} of $f_i$.
This discriminant has degree $n(d_i-1)^{n-1}$ and it vanishes
when the hypersurface $\{f_i = 0\}$ is singular in $\PP^{n-1}$.
If $| I | = n-1$ then $\Delta_I$ is the {\em mixed discriminant}~\cite{cattani2013mixed}
of a square polynomial system. Finally, if $| I | = n$ then $\Delta_I$ is the
{\em resultant} for $n$ hypersurfaces in $\PP^{n-1}$.
See \cite{Pok} for a recent study in this subject area and many relevant references.

\begin{example}[$n=3$]
Consider $m$ curves of degree at least two in $\PP^2$.
Then the Euler discriminant is a
product of $m + \binom{m}{2} + \binom{m}{3}$ irreducible polynomials
in $\sum_{i=1}^m \binom{d_i+2}{2}$ unknowns, namely the coefficients of
$f_1,f_2,\ldots,f_m$. These are the coordinates of the parameter space~$\PP$.

We now describe the factors in the Euler discriminant. 
For each index $i \in [m]$, we have the {\em discriminant} $\Delta_{\{i\}}$.
This vanishes
when the curve $\{f_i = 0\}$ is singular. 
It has degree $ 3 (d_i-1)^2$ in the coefficients of $f_i$.
For each pair of curves $f_i$ and $f_j$, we have the
mixed discriminant $\Delta_{\{i,j\}}$. This vanishes when
the curves are tangent, i.e.~they fail to intersect in
$d_i d_j$ isolated points. This mixed discriminant is known classically as the
{\em tact invariant}, and it has bidegree $(d_j^2 + 2 d_i d_j - 3 d_j,\,
d_i^2 + 2 d_i d_j - 3 d_i)$, by \cite[equation (1.4)]{cattani2013mixed}.
Finally, any triple of curves contributes 
a resultant which vanishes when the three curves intersect in $\PP^2$. 
The resultant  $\Delta_{\{i,j,k\}}$ has the
tridegree $(d_j d_k, d_i d_k, d_i d_j)$ in the coefficients of $f_i,f_j,f_k$.

To be even more specific, suppose we examine four quadrics
in $\PP^2$. Their  Euler discriminant is a polynomial
in $24=6+6+6+6$ unknowns. It is the product of
$14=4+6+4  $ irreducible discriminants $\Delta_I$. 
The $\ZZ^4$-degree of this Euler discriminant equals $(33,33,33,33)$.

If we allow some of the $f_i$ to be linear, as in Theorem \ref{thm:main2},
then the factor $\Delta_{\{i\}}$ disappears, because a line in $\PP^2$
can never degenerate and become singular. Thus, in Example \ref{ex:1plus3},
the Euler discriminant has only $13$ irreducible factors,
and its $\ZZ^4$-degree equals  $(18,25,25,25)$.
\end{example}

The  arrangement $\mathcal{A}$ is SNC if and only if
the Euler discriminant does not
vanish for the tuple $(f_1,f_2,\ldots,f_m) \in \PP$
that defines~$\mathcal{A}$. This notion of genericity is precisely
the condition which ensures that our matrix $\Qi^s$ in
(\ref{eq:ourmatrix}) has the desired commutative algebra properties.

\begin{remark}
The Newton polytope of $f_i$ can be strictly contained in $d_i \Delta_{n-1}$.
The arrangement  $\cA$ might be SNC even if this happens.
Under such a specialization, the discriminant $\Delta_I$ does not
vanish, but it encodes the discriminant
for the Newton polytopes of $f_i$, $i \in I$.
\end{remark}

The notion of the Euler discriminant was first introduced by Esterov \cite{Esterov}.
 Its role in likelihood geometry is the topic of
a recent article  by Telen and Wiesmann \cite{EulerStrati}.
A detailed analysis for hyperplanes
was undertaken by Fevola and Matsubara-Heo \cite{ClaudiaSaiei}.
Building on these advances, we show in 
Theorem~\ref{thm:Euler_discriminant_equals_mixed_discriminant} that
vanishing of the Euler discriminant for $\cA$ is equivalent to 
a drop in the Euler characteristic of
the arrangement complement $\PP^{n-1} \backslash \mathcal{A}$.

\section{The Kernel of a Matrix over a Ring}
\label{sec:det_ideals}

We here start the proof of Theorem~\ref{thm:main}, which says that
 the maximal minors of the matrix in (\ref{eq:ourmatrix}) 
 generate the likelihood ideal $I_{\cal A}$. Our strategy is to apply constructions from commutative algebra \cite{eisenbud2013commutative},
namely
Fitting ideals and the Buchsbaum--Rim complex, to the matrices $Q$, $\Qi$, and
$\Qi^s$. In fact, the results of this section show more: both the minor ideal
and $I_\cA$ define
vector bundles and are Cohen--Macaulay. To complete the proof of
Theorem~\ref{thm:main}, 
 it suffices to show that the
 determinantal ideal $I_m(\Qi)$ is reduced, which we tackle in Section \ref{sec:multidegrees}.

The Buchsbaum--Rim complex is most typically used when a matrix, or the module
it presents, is close to generic.
The matrices $Q$, $\Qi$, and $\Qi^s$ are  far from being generic.
 However, they are ``generic enough'' for the Buchsbaum--Rim construction to be applicable.

We recall some material from Eisenbud's
 textbook  \cite[\S 20.7 and \S A2.6]{eisenbud2013commutative}.
Let $R$ be a commutative Noetherian ring with unit, and $M$ a finitely generated $R$-module, with presentation
\[ \qquad
F \stackrel{\phi}{\longrightarrow}G \longrightarrow M \longrightarrow 0,\,\,\, \mbox{ where }F \simeq R^f \mbox{ and }G \simeq R^g.  \]
We write $I_g(\phi)$ for the ideal generated by the $g \times g$ minors of $\phi$.
Here $f$ and $g$ are positive integers with $f\ge g$. This is
the standard notation used in~\cite{eisenbud2013commutative}.
Fitting's lemma yields:

\begin{lemma}\label{FitLemma}
There exists a positive integer $k$ such that
$\,\ann(M)^k \subseteq I_g(\phi) \subseteq \ann(M)$.
Moreover, the {\em Fitting ideal} $I_g(\phi)$ is independent of the choice of presentation
of the module~$M$.
\end{lemma}

Our next ingredient is the {\em Buchsbaum--Rim complex}. We state this construction for characteristic zero.
In positive characteristic it is necessary to replace $\Sym$ with divided~powers. 

\begin{proposition}\label{BRcomplex}
The following sequence is a complex of $R$-modules:
\begin{equation}
        \label{eq:BuchsbaumRim}
       0 \longrightarrow \Sym^{f-g-1}G \otimes \bigwedge^f F \rightarrow \cdots \rightarrow G \otimes \bigwedge^{g+2} F \longrightarrow \bigwedge^{g+1}F\longrightarrow  F\stackrel{\phi}{\longrightarrow} G \longrightarrow M \longrightarrow 0.
    \end{equation}
    Writing $e_i$ for elements of $F$,
the first differential from $\bigwedge^{g+1}F$ to $F$ is the $R$-linear map
\[
e_1\wedge e_2 \wedge \cdots \wedge e_{g+1} \,\,\,\mapsto \,\,\,\sum\limits_{i=1}^{g+1} (-1)^i e_i \cdot \wedge^g(\phi)(e_1\wedge e_2 \wedge \cdots \wedge \widehat{e_i} \wedge \cdots \wedge e_{g+1}).
\]
Similar formulas for the higher differentials are found in \cite[\S A2.6]{eisenbud2013commutative}.
 This complex is exact if and only if the Fitting ideal $I_g(\phi)$ contains a regular sequence of length $f-g+1$.
\end{proposition}

We shall prove that the maximal minors of the matrices $\Qi$, and $\Qi^s$ contain long enough regular sequences so that their Buchsbaum--Rim complexes are in fact resolutions.
 The $\binom{m+n-1}{m+1}$ syzygies which generate $\ker(\Qi)$ are then given abstractly by Proposition~\ref{BRcomplex}. 
 
 \smallskip
 
 We prefer to write these syzygies in a combinatorial notation that is easy to understand.
Think of $\phi$ as a matrix with $g$ rows and $f$ columns. Every submatrix given by
$g+1$ columns has a distinguished vector in its kernel.
Linear algebra students learn this as {\em Cramer's~rule}.
The coordinates of that vector are the signed $g \times g$ minors of the submatrix.
For combinatorialists, these vectors correspond to the {\em circuits} of the matroid of $\phi$.
Indeed, the circuits in ${\rm ker}(\phi)$ are the
  non-zero vectors whose support is inclusion-minimal, and these span the kernel.
  This is familiar when $R$ is a field. 
  But, it actually works over any ring, assuming
the hypothesis in  Proposition \ref{BRcomplex} is satisfied.
We formulate the take-home message as follows:

  \begin{corollary}[Buchsbaum--Rim]
  The kernel of a matrix is generated by the circuits of~the matrix, provided
   the ideal of maximal minors contains a regular sequence that is long enough.
  \end{corollary}
    
  To encode vectors in $F \simeq R^f$, we introduce
  a vector $s = (s_1,s_2,\ldots,s_f)$ of formal variables,
  and we employ linear forms in $s$ with coefficients in $R$.
  The element of $F$ represented by~such a linear form is its
    gradient with respect to $s$.
    Let $\phi^s$ denote the matrix with $g+1$ rows and $f$ columns
    obtained by augmenting $\phi$ with the row vector $s$.
    With this encoding, we have:
    
\begin{remark}
  \label{rem:detsyz}
    The circuits of the matrix $\phi$ are the maximal minors of
    the augmented matrix $\phi^s$. These generate the kernel of $\phi$
    whenever a long enough regular sequence can be found.
\end{remark}
   
We now return to the task of proving Theorem~\ref{thm:main}. To make use of the full strength of Proposition~\ref{BRcomplex}, our next step is to analyze the matrices $Q$, $\Qi$ and $\Qi^s$ in more detail. We start by showing that the kernel of the matrix $Q$ defines a rank $n$ vector bundle on $\PP^{n-1}$.
Our assumption that intersections of the hypersurfaces $\V(f_i)$ are well-behaved will be crucial. 

\begin{lemma}
    \label{lem:rank_jacobian}
    Let  $\cA$ be SNC, $I \subset [m]$ and 
    $X_I := \V(f_i \mid i \in I) \backslash \bigcup\limits_{j \not \in I } \V(f_j)$.
 At any point $p\in X_I$, the rank of the Jacobian $J_I$ 
    of $\{f_i:i\in I\}$ satisfies $ \,  \rk(J_I(p)) =  {\rm min}(|I|,n)$.
\end{lemma}

\begin{proof}
This follows from the SNC property.
    If $|I| \geq n$ then $X_I \subset \PP^{n-1}$ is empty.
Otherwise,
         \[
        \dim T_p(X_I) \,\,=\,\, \dim(X_I) \,\,=\,\, \dim \ker(J_I(p)) \,\,=\,\, n-| I |. \qedhere
    \]
\end{proof}

\begin{lemma}\label{lem:kerQ_const}
    With the assumptions of \Cref{lem:rank_jacobian}, 
    the matrices $Q$ and $\Qi$ have constant rank
as        $p$ ranges over all of $\,\PP^{n-1}$.
The kernels of  both matrices are vector bundles on $\PP^{n-1}$.
\end{lemma}

\begin{proof}
    If $p \in \PP^{n-1}\backslash \V(f)$, then $Q(p)$ is in row-echelon form with nonzero pivots.
    Its kernel has dimension $n$ since    $\rk(Q(p)) = m$.
    If $p\in \V(f)$, upon reordering the $f_i$, we may assume that $p\in X_I$ for $I = \{1,\ldots,i\}$
    with $i < n$ and
     $X_I$ as in Lemma~\ref{lem:rank_jacobian}. We divide $Q$ into blocks,
    \[        Q \,=\, \begin{pmatrix}
            Q_{11} & Q_{12} & Q_{13} & Q_{14} & Q_{15} \\
            Q_{21} & Q_{22} & Q_{23} & Q_{24} & Q_{25} \\
            Q_{31} & Q_{32} & Q_{33} & Q_{34} & Q_{35} \\
        \end{pmatrix}, \]
    where the column blocks have sizes $(i, n-i, m-n, i, n-i)$, while the row blocks have sizes $(i,n-i,m-n)$. The assumption $p\in X_I$ implies that, after substituting $p$, we have
    \[
        Q_{11}(p) \,=\, Q_{12}(p) \,=\, Q_{13}(p)\, =\, Q_{21}(p)\, =\, Q_{23}(p)\, = \,Q_{31}(p)\, =\, Q_{32}(p) \,\,=\,\, 0.
    \] 
    Moreover, $Q_{22}(p)$ and $Q_{33}(p)$ are diagonal matrices with non-zero diagonal entries. By \Cref{lem:rank_jacobian}, the $i\times n$ matrix $\bigl(Q_{14}(p)\mid Q_{15}(p)\bigr)$ defining the Jacobian $J_I(p)$ has rank $i$. After reordering variables, we may assume that $Q_{14}(p)$ is an invertible matrix; using it we may row reduce so that $Q_{24}(p) = Q_{34}(p)=0$. Similarly, we may assume that $Q_{25}(p)$ is invertible and $Q_{35}(p) = 0$. Again, we conclude $\rk(Q(p)) = m$ and thus $\dim(\ker Q(p)) = n$.

    We pass from $Q$ to $\Qi$ by deleting the first column. The key here is to notice that 
    \[
    \coker(Q) \,\,\simeq \,\, \coker(\Qi).
    \] 
    This holds because $(d_1,d_2,\ldots,d_m,-x_1,\ldots,-x_n)$
    is in the kernel of $Q$. Since $d_1$ is a positive integer, it is a unit in $R$, and removing
the first column of $Q$ does not change the cokernel. 
\end{proof}

\begin{remark}\label{rem:defK}
We noted in the introduction
  that the likelihood ideal always contains the $(1,0)$-form $\sum d_is_i$. Passing from $Q$ to $\Qi$ lets us focus on the generators of the likelihood ideal of degree $(1,a)$ with $a > 0$. From now on, we write $K$ for the ideal $I_{m+1}(\Qi^s)+\langle 
  \sum_i d_is_i \rangle$. 
\end{remark}
Combining the results above, we arrive at the following conclusion, which is also obtained in \cite{TohaneanuEtAl2025}, via a different method.

\begin{proposition}\label{circuitSyzgenerate}
If $\cA$ is SNC 
     then the Buchsbaum--Rim complex for $\Qi$ is~exact. 
\end{proposition}

\begin{proof}
By Lemma~\ref{lem:kerQ_const}, the cokernel of $\Qi$ is supported on the irrelevant ideal of $\PP^{n-1}$.
Hence
\[
\sqrt{\ann \coker(\Qi)} \,\,=\,\,\langle x_1, \ldots, x_n\rangle.
\]
By Lemma~\ref{FitLemma}, this means that also 
\[\sqrt{I_m(\Qi)} \,\,=\,\,\langle x_1, \ldots, x_n\rangle.\] Therefore $I_m(\Qi)$ contains a regular sequence of length $(m+n-1)-m +1 =n$,
namely $x_1^N,\ldots,x_n^N$ for  $N \gg 0$. This ensures 
 the exactness of the Buchsbaum--Rim complex. 
\end{proof}
\begin{lemma}\label{lem:minorsIdealVB}
The ideal
    $I_{m+1}(\Qi^s)$ has codimension $n-1$ in $\CC[s_2,\ldots,s_m,x_1,\ldots,x_n]$. It defines a rank $m-n$ vector bundle on $\PP^{n-1}$. The variety $\,\V(K)$ is irreducible of codimension~$n$. 
\end{lemma}

\begin{proof}
    The proof is similar to that of Lemma~\ref{lem:kerQ_const}. We adopt the same notation,  except     
    that we work with $\Qi^s$, so $p\in X_I$ means $p\in \V(f_2, \ldots, f_i) \backslash \! \cup_{j>i}\V(f_j)$.
    We divide $\Qi^s$ into blocks as in Lemma~\ref{lem:kerQ_const}.
    Substituting $p$ into $\Qi^s$ and simplifying as in Lemma~\ref{lem:kerQ_const} yields the matrix
    \[
        \Qi^s(p) \,\,=\,\, \begin{pmatrix}
            0_{i \times (i-1)} & 0_{i \times (n-i)} & 0_{i \times (m-n)} & {\rm Id}_i & 0_{i \times (n-i)}  \\
            0_{(n-i) \times (i-1)} & {\rm Id}_{n-i} & 0_{(n-i) \times (m-n)}  & 0_{(n-i) \times i}  & {\rm Id}_{n-i} \\
            0_{(m-n) \times (i-1)} & 0_{(m-n) \times (n-i)}  & {\rm Id}_{m-n} & 0_{(m-n) \times i}  & 0_{(m-n) \times (n-i)}  \\
            [s_2,\dotsc, s_i] &[s_{i+1},\dotsc, s_n]& [s_{n+1},\dotsc, s_m]& 0_{1 \times i} & 0_{1 \times (n-i)}  
        \end{pmatrix}.
    \]
    Note that the number of columns has dropped by one, and the number of rows increased by one.
     Hence the column blocks have sizes $(i-1, n-i, m-n, i, n-i)$, and the row blocks have sizes $(i,n-i,m-n,1)$, respectively. A short computation shows $I_{m+1}(\Qi^s(p)) = \langle s_2,\ldots ,s_n\rangle$. 
    
  At first glance this seems to define a trivial bundle, but this is a reflection of the modifications made to $\Qi^s(p)$
    in order to simplify our local calculation. We conclude that the ideal $I_{m+1}(\Qi^s)$
    has codimension $n-1$ in $ \CC[s_2,\ldots,s_m,x_1,\ldots,x_n]$  and that $\V(I_{m+1}(\Qi^s))$ defines a vector bundle on $ \PP^{n-1}$, so is irreducible. Since $s_1$ does not appear in $\Qi^s$, adding  $\sum d_is_i$ to $I_{m+1}(\Qi^s)$ yields an ideal of codimension $n$ in the larger ring $S$, and concludes the proof. 
    \end{proof}

\begin{corollary}\label{cor:minorsEN}
$I_{m+1}(\Qi^s)$ has an Eagon--Northcott resolution, so it is Cohen--Macaulay.
\end{corollary}
\begin{proof}
Arguing exactly as in the proof of Proposition~\ref{circuitSyzgenerate} shows that $I_{m+1}(\Qi^s)$ contains a regular sequence of length $(m+n-1)-(m+1)+1 = n-1$. Hence, by \cite[Theorem A2.10]{eisenbud2013commutative}, the Eagon--Northcott complex is exact and the Cohen--Macaulay property follows.
\end{proof}

We now reconnect with likelihood geometry.
The ideal $I_\mathcal{A}$ of the likelihood correspondence may be described as follows: fix the localization $S_f$ of $S$ at 
the product $f=f_1 f_2 \cdots f_m$ which defines $\mathcal{A}$.
The partial derivatives of the log-likelihood function
generate the prime~ideal
\begin{equation}
\label{eq:idealwithdenominators}
\biggl\langle \,\sum_{i=1}^m \frac{s_i }{f_i} \cdot \frac{\partial f_i}{\partial x_j} \,\,:
\,\, j = 1,2,\ldots,n \,\biggr\rangle \,\,\,\, \subset \,\,\,S_f.
\end{equation}
Then $I_\mathcal{A}$ is the intersection of (\ref{eq:idealwithdenominators})
with the polynomial ring $S = \CC[s_1,\ldots,s_m,x_1,\ldots,x_n]$.

\begin{lemma}~\label{lem:localEqual}
The ideal of $S_f$ that is generated by the maximal minors of $\Qi^s$ equals~(\ref{eq:idealwithdenominators}).
\end{lemma}

\begin{proof}
We first note that the linear form $\sum_{i=1}^m d_i s_i$ is contained in both ideals.
We can therefore modify $\Qi^s$ by adding the missing first column  and by dividing the $i^{th}$ row by $f_i$:
\begin{equation}
\label{eq:Fmatrix}
\tilde{Q^s} \,\,\,= \,\,\,
     \begin{pmatrix}
        \, f_1^{-1} & 0 & \cdots & 0 & 0 \\
        \, 0 & f_2^{-1} & \cdots & 0 & 0 \\
        \, \vdots & \vdots & \ddots & \vdots & \vdots \\
        \, 0 & 0 & \cdots & f_m^{-1} & 0 \\
        \, 0 & 0 & \cdots & 0 & 1 
        \end{pmatrix} 
        \begin{pmatrix}
        \, f_1 & 0 & \cdots & 0 & \frac{\partial f_1}{\partial x_1} & \frac{\partial f_1}{\partial x_2} & \cdots & \frac{\partial f_1}{\partial x_n\ } \smallskip
\\
 \, 0 &       f_2 & \dots & 0 & \frac{\partial f_2}{\partial x_1} & \frac{\partial f_2}{\partial x_2} & \cdots & \frac{\partial f_2}{\partial x_n\ } \smallskip \\
   \,    0 & \vdots & \ddots & \vdots & \vdots & \vdots & \ddots & \vdots \smallskip \\
\,         0 & 0 & \cdots & f_m & \frac{\partial f_m}{\partial x_1} & \frac{\partial f_m}{\partial x_2} & \cdots & \frac{\partial f_m}{\partial x_n\ }  \smallskip \\
\, s_1 & s_2 & \cdots & s_m & 0 & 0 & \cdots & 0
    \end{pmatrix}.
\end{equation}
We claim that (\ref{eq:idealwithdenominators}) coincides with the ideal 
generated by the maximal minors of $\tilde{Q^s}$. One inclusion
follows from the observation that
the maximal minor of $\tilde{Q^s}$ with column indices $1,2,\ldots,m$ and $m+j$
is precisely  the $j^{th}$ generator in (\ref{eq:idealwithdenominators}), after
multiplication by $-1$.

For the converse, we return to the general linear algebra setting around
  Remark~\ref{rem:detsyz}.
Consider an $n \times (m+n)$ matrix that has the form
$P = \bigl( \,{\rm Id}_{m}\,| \, A \,\bigr)$,
where $A$ is any $m \times n $ matrix with entries in a ring $R$.
We claim that the kernel of $P$ is the free $R$-module
generated by the rows of the $n \times (m+n)$ matrix
$Q = \bigl( \,-A^T \,|\, \,{\rm Id}_{n}  \,\bigr)$.
Indeed, we clearly have $P \cdot Q^T = 0$. Moreover,
 if $(u,v)^T$ is any vector in the kernel of $P$, then
$u + Av = 0$ and hence $(u,v)^T = Q^Tv$.

We apply the previous paragraph to  $R = S_f$
where $P$ consists of the first $m$ rows of
$ \tilde{Q^s}$. Any kernel vector $(u,v)$
 is an $R$-linear combination
of the $n$ row vectors of $Q$. These are
$$ \biggl(
-\frac{1}{f_1} \frac{\partial f_1}{\partial x_j},
-\frac{1}{f_2} \frac{\partial f_2}{\partial x_j},
\,\ldots\,,
-\frac{1}{f_m} \frac{\partial f_m}{\partial x_j},
\,0,\ldots,0,1,0,\ldots ,0 \biggr) \quad
{\rm for}\,\,\, j=1,2,\ldots,n.
$$
In particular, every circuit syzygy of the first $m$ rows of $\tilde{Q^s}$ has this property.
We conclude that every maximal minor of $\tilde{Q^s}$ is an $S_f$-linear
combination of the generators in (\ref{eq:idealwithdenominators}).
\end{proof}

The following proposition summarizes the results we have proved in this section.

\begin{proposition}\label{prop:idealSummary}
    The ideal $K$ is Cohen--Macaulay of codimension $n$, and $\sqrt{K} = I_{\cal A}$.
\end{proposition}
\begin{proof}
By Lemma~\ref{lem:minorsIdealVB} and Corollary~\ref{cor:minorsEN}, $I_{m+1}(\Qi^s)$ is primary to the codimension $n-1$ prime ideal $\sqrt{I_{m+1}(\Qi^s)}$. The linear form $\sum_i d_is_i$ involves a variable $s_1$ which does not appear in $I_{m+1}(\Qi^s)$, so is a non-zero divisor on $S/I_{m+1}(\Qi^s)$ and hence preserves the Cohen--Macaulay property, while increasing the codimension by one. Lemma 2.4 of \cite{ArrangementsAndLikelihood} shows that $I_{\cal A}$ is prime of codimension $n$. The result follows because
$K= I_{\cal A} $ in $S_f$,
by Lemma~\ref{lem:localEqual}.
\end{proof}

Proposition~\ref{prop:idealSummary} falls just short of being a full proof of Theorem~\ref{thm:main}: while $K$ is $I_{\cal A}$-primary, it need not be reduced. The next section introduces the tools needed to complete the proof of Theorem~\ref{thm:main}. Finally, Theorem~\ref{thm:main2} is just a variant of Theorem~\ref{thm:main}.

\section{Multidegrees and Euler Discriminant}
\label{sec:multidegrees}

The goal of this section is twofold. First, we finish the proof 
started in Section \ref{sec:det_ideals}.
This is done by showing that $K = I_{m+1}(Q^s_{\backslash 1})+ 
\bigl\langle \sum_{i=1}^m d_i s_i \bigr\rangle$ 
has the same multidegree as the likelihood ideal $I_\cA$. That multidegree 
is given in Lemma~\ref{lem:generic_multidegree}.
Second, we justify our use of Euler's name,
by Theorem \ref{thm:Euler_discriminant_equals_mixed_discriminant}.
Namely,
we prove that vanishing of the Euler discriminant,
i.e.~the product $\prod_I \Delta_I$ in
Section~\ref{sec:main_result}, is equivalent to a drop in the Euler chararacteristic of
 $\PP^{n-1} \backslash \cA$.

The ideals $I_\cA$ and  $K$ 
are subsets of the polynomial ring $S$. They define subvarieties of
$\PP^{m-1} \times \PP^{n-1}$. The Chow ring of this product space is the truncated polynomial~ring $A^*(\PP^{m-1}\times \PP^{n-1}) = \ZZ[\sigma,\tau]/\langle\sigma^m,\tau^n\rangle$, where $\sigma$ and $\tau$ are hyperplane classes in $\PP^{m-1}$ and~$\PP^{n-1}$.

 The \emph{multidegree} $[I]$ of a
 bihomogeneous ideal $I$ in $S$ is the class of $\V(I)$ in $A^*(\PP^{m-1}\times \PP^{n-1})$. 
 If $\V(I)$ has codimension $n$ then the multidegree $[I]$
    is a binary form in $\sigma$ and $\tau$ of degree $n$.
    We consider this binary form    for the prime ideal $I_\cA$ of the likelihood correspondence $\mathcal{L}_\cA$:
\begin{equation}
    \label{eq:multidegree_binary_form}
    [I_\cA] \,\,=\,\, c_0\sigma^{n} + c_1\sigma^{n-1}\tau + c_2\sigma^{n-2}\tau^{2} + \dots + c_{n-1}\sigma\tau^{n-1}.
\end{equation}

Note that $\tau^n=0$ in $A^*(\PP^{m-1}\times \PP^{n-1})$.
For any arrangement $\cA$, 
the coefficient $c_j$ with $j = \max\left\{ j\in \left\{ 0,\dots,n-1 \right\} \midcolon c_j\neq 0 \right\}$ is the \emph{ML degree}; see \cite[Definition 3.1]{ArrangementsAndLikelihood}. 
This is $c_{n-1}$ when $\cA$ is SNC.
The ML degree is the
Euler characteristic of the arrangement complement $\PP^{n-1} \backslash \cA$,
 as was proved in \cite[Theorem~20]{catanese2006maximum} and generalized in 
 \cite[Theorem~1]{huh2013maximum}. Each coefficient $c_i$ is the Euler characteristic of the intersection of $\PP^{n-1} \backslash \cA$ with a general linear space of dimension $i$ in $\PP^{n-1}$. 
The  multidegree $[I_\cA]$ was computed    in \cite[Theorem 1]{catanese2006maximum}
for generic arrangements $\cA$.
 We restate this result,    with its genericity assumption paraphrased.

\begin{lemma}
    \label{lem:generic_multidegree}
    If $\cA$ is SNC then the multidegree $[I_\cA]$ is given by a simple
    generating function. Namely, the coefficient $c_i$ in
          (\ref{eq:multidegree_binary_form}) is
 the coefficient of $z^i$ in the Taylor series expansion of
    \begin{equation}
        \label{eq:generic_multidegree}
        \frac{(1-z)^n}{(1-d_1z)(1-d_2z)\,\cdots \,(1-d_mz)}~.
    \end{equation}
\end{lemma}

\begin{example}[$n=3$]
    For $m$ curves in $\PP^2$, the multidegree of the likelihood ideal  equals
    $$ [I] \quad = \quad \sigma^3 \,\,+\,\, \left(\,\sum_{i=1}^m d_i - 3\right) \,\sigma^2 \tau \,\,+ \,\,
    \left( \sum_{1 \leq i \leq j \leq m} \!\!\! d_i d_j \, -\, 3 \sum_{i=1}^m d_i \,+\, 3 \right) \,\sigma \tau^2. $$
    The expression in the parentheses on the right is the ML degree.
    This is $13$ in Example \ref{ex:1plus3}.
\end{example}

We already know from the previous section that our ideal $K$
is contained in $I_\cA$ and that both ideals have codimension $n$.
We now show that they also have the same multidegree.

\begin{lemma}
    \label{lem:multidegrees_agree}
    The multidegree $[I_\cA]$ equals the multidegree $[K]$
    of the determinantal ideal $K$.
\end{lemma}

\begin{proof}
    We first compute $[K']$ where $K' = I_{m+1}(\Qi^s)$. 
    We record the bidegree of
    each entry in~$\Qi^s$. Bidegrees are linear forms in $\sigma$
    and~$\tau$, representing divisor classes in $A^*(\PP^{m-1}\times
    \PP^{n-1})$:
    \begin{equation}
        \label{equ:gradingMatrix}
        \begin{pmatrix}
          d_1\tau & \cdots & d_1\tau & (d_1-1)\tau & \cdots & (d_1-1)\tau \\
          \vdots & \ddots & \vdots & \vdots & \ddots & \vdots \\
          d_m\tau & \cdots & d_m\tau & (d_m-1)\tau & \cdots & (d_m-1)\tau \\
          \sigma & \cdots & \sigma & \sigma - \tau & \cdots & \sigma - \tau \\
        \end{pmatrix}.
    \end{equation}
    
     With this bigrading, every maximal minor of  the 
    $(m+1) \times (m+n-1)$ matrix $\Qi^s$
    is homogeneous. We can thus interpret $\Qi^s$
     as a map $\varphi$ of vector bundles over $\PP^{m-1}\times \PP^{n-1}$:
         \[
        \varphi\,\,\colon\,\, \cE \,=\, \cO(-1,0) \,\oplus \,\bigoplus_{i=1}^m \cO(0,-d_i) 
        \,\,\,\rightarrow \,\,\,\cF \,=\, \cO^{\oplus m-1} \,\oplus\, \cO(0,-1)^{\oplus n}.
    \]
    Here, $\cO$ is the trivial bundle on $\PP^{m-1} \times \PP^{n-1}$ and $\cO(d,e) = \cO_{\PP^{m-1}}(d) \boxtimes \cO_{\PP^{n-1}}(e)$. Since we are taking the maximal minors of 
$Q^s_{\backslash 1}$, we are looking for the rank-$m$ locus of $\varphi$. As $K = K' + \langle d_1s_1 + \dots + d_ms_m\rangle$ and $\codim(\V(K)) = n$, we have $\codim(\V(K')) = n-1$. We now
apply  the Giambelli--Thom--Porteous Formula \cite[Theorem 12.4]{3264}, which states that
    \[
        [K'] \,\,=\,\, \Delta^1_{n-1}\left( \frac{c(\cF)}{c(\cE)} \right),
    \]
    where $c$ is the total Chern class and $\Delta^1_{n-1}$ is a double Schur polynomial.
    In our case, the~right hand side is simply the
      degree $n-1$ part of $c(\cF)/c(\cE)$. The ratio of Chern classes equals
    \[
        \frac{c(\cF)}{c(\cE)} \,\,=\,\, \frac{(1-\tau)^n}{(1-\sigma) \prod_{i=1}^m (1-d_i\tau)}
        \,\, =\,\, \left(\sum_{j\geq 0} \sigma^j\right) \frac{(1-\tau)^n}{(1-d_1\tau)(1-d_2\tau)\cdots(1-d_m\tau)}.
    \]
    The degree $n-1$ part of this expression equals
$
        c_0\sigma^{n-1} \,+\,c_1 \sigma^{n-2}\tau \,+ \, c_2 
        \sigma^{n-3}\tau^2 \,+ \,\cdots \,+ \,c_{n-1} \tau^{n-1}$,
    with the $c_i$ as in \eqref{eq:generic_multidegree}. After intersecting with the transverse linear space $\V(d_1s_1 + \dots + d_ms_m)$, the multidegree above is multiplied by $\sigma$.
    Thus, the multidegree $[K]$ agrees with $[I_\cA]$.
\end{proof}

\begin{proof}[Proof of Theorems \ref{thm:main} and \ref{thm:main2}]
The ideal $K$ is contained in the likelihood ideal $I_\cA$. The former is
 Cohen--Macaulay, the latter is prime, and they  have same dimension.
  They share the same multidegree, by Lemma  \ref{lem:multidegrees_agree}.
  This implies that they must be equal, i.e.~$I_\cA = K$.
\end{proof}

We now turn to our second aim, namely to further elucidate the SNC hypothesis.
The parameter space $\PP$ for homogeneous
polynomials $f_1,\ldots,f_m \in \CC[x_1,\dots,x_n]$ of degrees $d_1,\dots,d_m$ is the product 
in (\ref{eq:parameterspace}).
For any subset $I = \{i_1,\dots,i_k\}$ in $[m] = \{1,2,\ldots,m\}$, 
the \emph{mixed discriminant variety} $\nabla_I$ is the Zariski closure of the set of degenerate $k$-tuples:
\begin{align*}
    \nabla_I \,\,&=\,\, \mathrm{cl}_\text{Zariski} \big\{ (f_{i_1},\dots,f_{i_k}) \in \PP^{\binom{d_{i_1}+n-1}{n-1} - 1} \times \dots \times \PP^{\binom{d_{i_k}+n-1}{n-1} - 1} \midcolon \\
    & \quad ~\V(f_{i_1},\dots, f_{i_k}) ~\text{is not smooth of codimension } k \big\} \subset \PP^{\binom{d_{i_1}+n-1}{n-1} - 1} \times \dots \times \PP^{\binom{d_{i_k}+n-1}{n-1} - 1}.
\end{align*}

This variety is either empty or a hypersurface. In the latter case, the mixed discriminant $\Delta_{I}$ is the defining polynomial of $\nabla_{I}$; otherwise one sets $\Delta_{I} = 1$.
See \cite{cattani2013mixed, ClaudiaSaiei, gelfand2008discriminants, Pok} for  details. 
The Euler discriminant
    $\Delta$ of the arrangement $\cA = \left\{ f_1,\dots,f_m \right\} $ was defined as the product
    \[        \Delta\,\, = \prod_{\substack{I\subset [m] \\ |I|\leq n}} \Delta_{I}.    \]
In what follows, this is called the {\em algebraic Euler discriminant}. We
write $\nabla = \V(\Delta)$ for the  reducible hypersurface it defines in the
parameter space $\PP$.
Our arrangement $\cA = \V(f)$ is SNC if and only if
$\cA \notin \nabla$ if and only if the algebraic Euler discriminant $\Delta$ is nonzero for $\cA$.
For extra clarity, we remind the reader that the notation $\cA \not\in \nabla$ means that the
vector of all coefficients of $f_1,\ldots,f_m$ is a point in $\PP$
that lies outside the hypersurface  $\nabla = \V(\Delta)$.

In the setting of \cite{gelfand2008discriminants},
 one would now refer to the
Cayley polytope of the scaled simplices $d_i \Delta_{n-1}$ 
for $i=1,2,\ldots,m$. This polytope has dimension $n+m-2$
and $mn$ vertices. Writing $A$ for this set of vertices, we consider
the principal $A$-determinant $E_A$ from
 \cite[Chapter 10]{gelfand2008discriminants}.
This is the product of the
discriminants over all faces of the Cayley polytope ${\rm conv}(A)$.

It follows from the definitions that the Euler discriminant $\Delta$
divides the principal $A$-determinant~$E_A$. 
Every factor of $\Delta$ arises from some 
   face of the Cayley polytope ${\rm conv}(A)$, but not all faces are needed.
   In other words, $E_A$ has more factors than~$\Delta$.
To be concrete, in the following example,
$\Delta$ is a polynomial of degree $12$, while
$E_A$ is a polynomial of degree~$60$.

\begin{example}[Four lines in $\PP^2$]\label{ex:4linesP2}
Let $n=3$ and $d_1=d_2=d_3=d_4=1$. The
Cayley polytope is $A = \Delta_2 \times \Delta_3$, which has
normalized volume $10$.
We write $C$ for the $4 \times 3$ matrix of
coefficients of the linear forms $f_1,\ldots,f_4$.
The entries of $C$ are coordinates on $\PP = (\PP^2)^4$.
The Euler discriminant $\Delta$ is the product of the four
$3 \times 3$ minors of $C$, so $\Delta$ is a polynomial
of degree $12$. The principal $A$-determinant $E_A$ is
the product of all minors of $C$, including those of
size $1 \times 1$ and $2 \times 2$. Hence $E_A$
is a polynomial of degree $60 = 6 \cdot 10 =
({\rm dim}(A)+1) \cdot {\rm vol}(A)$.
A general formula for $E_A$ in the case of hyperplane arrangements
is found in \cite[Theorem~3.9]{ClaudiaSaiei}.
\end{example}

A \emph{topological}
notion  of Euler discriminant, in the spirit of Esterov \cite{Esterov}, was studied recently in
 \cite{fevola2024landau} and in \cite{EulerStrati}. This Euler discriminant variety 
$\nabla_{\chi}$ is the Zariski closure of the parameter locus where the topological Euler characteristic 
$\chi$ of an arrangement complement differs from the generic Euler characteristic \cite[Definition 4.1]{EulerStrati}. 
The two notions agree in our context.

\begin{theorem}
    \label{thm:Euler_discriminant_equals_mixed_discriminant}
    The algebraic Euler discriminant $\nabla$  agrees with the topological Euler discriminant $\nabla_\chi$ of the     projective hypersurface family complement defined by $\,f = f_1f_2 \cdots f_m$.
\end{theorem}

\begin{proof}
Applying the Cayley trick (see \cite[\S 2]{cattani2013mixed},
\cite[(2.5)]{ClaudiaSaiei}) we consider the auxiliary polynomial
$$\tilde{f}(x,y) \,= \,y_1f_1(x) + y_2f_2(x) + \dots + y_mf_m(x) \,\in\, R[y_1,\dots,y_m],
    $$ 
    where $y_1,y_2,\ldots,y_m$ are new unknowns.
    Thus $\tilde{f}(x,y)$ is a polynomial
        supported on the Cayley polytope of $f_1,f_2,\ldots,f_m$.
        In the notation of Section \ref{sec:main_result}, this is
         ${\rm conv}(\cA_I)$ for $I = [m]$.

Let $(\CC^*)^{n-1}$ be the dense torus in $\PP^{n-1}$, given by
$x_1 x_2 \cdots x_n \not= 0$. We also write
 $(\CC^*)^{m+n-2}$ for the dense torus in $\PP^{m-1} \times \PP^{n-1}$,
 given by $y_1 y_2 \cdots y_m x_1 x_2 \cdots x_n \not= 0$.
Theorem 2.2 in~\cite{ClaudiaSaiei} states that $(\CC^*)^{n-1} \backslash \cA$ and
$(\CC^*)^{m+n-2} \backslash V(\tilde{f})$
have the same Euler characteristic, up to sign.

This reduces our statement to the case of hypersurface families,
which was studied in detail in \cite{EulerStrati}, albeit for hypersurfaces
in a torus.
By Theorem 5.1 in \cite{EulerStrati}, the topological
Euler discriminant of the family given by $\tilde f$ is
precisely the principal $A$-determinant~$E_A$.
This shows that the algebraic Euler discriminant equals
the topological Euler discriminant, but for the
arrangement of $m+n$ hypersurfaces in $\PP^{n-1}$
given by $f_1,f_2,\ldots,f_m$ and $x_1,x_2,\ldots,x_n$.

For a generic point $\cA\in\nabla_I$ on some mixed discriminant, the hypersurface $V(\tilde{f})$ is singular \cite[Theorem 2.1]{cattani2013mixed}, so the Euler characteristic of its complement differs from its generic value by the Milnor number of the singularity. This singularity generically lies off the coordinate hyperplanes $x_1,x_2,\dots,x_n$, so the change in Euler characteristic can be observed both in $\bigl(\PP^{n-1} \times (\CC^*)^{m-1}\bigr)\backslash V(\tilde{f})$ and in $(\CC^*)^{m+n-2}\backslash V(\tilde{f})$. This shows the containment $\nabla \subseteq \nabla_\chi$.

Conversely, if $\cA$ is a generic point on $\nabla_\chi$, its Euler characteristic is not equal to its generic value. Since its absolute value is the leading coefficient of the multidegree $[I_\cA]$, by Lemmas \ref{lem:generic_multidegree} and \ref{lem:multidegrees_agree}, $\cA$ cannot be SNC, so $\cA\in \nabla$. This concludes the proof, showing $\nabla =  \nabla_\chi$.
\end{proof}

We now drop the adjectives ``algebraic'' and ``topological''
in front of ``Euler discriminant''. There is only one Euler discriminant $\nabla$. It characterizes the SNC condition for~$\cA$.


\section{Modules Associated to Arrangements}
\label{sec:homological}

The goal of this section is to connect likelihood geometry to the various modules that are
commonly associated to arrangements of hyperplanes and hypersurfaces. 
Recall that the
module of $R$-derivations is the free $R$-module $\Der(R)$ with basis
$\{\partial_1, \dotsc, \partial_n\}$.  
The \emph{module of derivations tangent to the arrangement $\cA$} is the submodule 
\begin{equation}
    \label{eq:def_Der_module}
    \Der(\cA)\,\, =\,\, \{\theta \in \Der(R) \midcolon \theta (f) \in 
        \langle f\rangle \}.
\end{equation}
See e.g.\ \cite{HalfreeCurves} for the basic definitions. If $f_1,\dots,f_m$ are coprime (e.g.\ if $\cA$ is SNC), we may 
replace \eqref{eq:def_Der_module} with the requirement that
$\theta(f_i) \in \langle f_i \rangle$ for all $i=1,\ldots,m$.
Lemma~2.2 in \cite{ArrangementsAndLikelihood} makes this more explicit. It displays several
modules which are all isomorphic to $\Der(\cA)$.

Derivations in $\Der(\cA)$ are closely related to 
  the likelihood correspondence~$\mathcal{L}_\cA$.
This was phrased in \cite[Theorem~2.11]{ArrangementsAndLikelihood} as an $R$-module homomorphism 
$\varphi  : \Der(\cA) \to \left( I_\cA \right)_{(1,\bullet)}$. Here,
derivations $\theta$ are mapped to their
evaluation  $\theta(\ell_\cA)$ at the log-likelihood function $\ell_\cA$.
The latter is a polynomial, and it lies
 in the likelihood ideal $I_\cA$, since
$f_i$ divides $\theta (f_i)$.
The subscript $(1,\bullet)$ denotes the part of degree one in the $s$ variables.
The map $\varphi$ is injective, and \Cref{thm:main} shows that it is an isomorphism 
when  the arrangement $\cA$ is SNC.

We now explain how to construct preimages of circuit syzygies under
the isomorphism $\varphi$. To this end, we replace the
bottom right zeros in $\Qi^s$ with the rational functions 
$\partial \ell_\cA/\partial x_j$:
   \begin{equation}
   \label{eq:Qsdelta}
    Q^{s,\partial}_{\backslash 1} \,\,=\,\, \begin{pmatrix}
        0 & \dots & 0 & \frac{\partial f_1}{\partial x_1} & \frac{\partial f_1}{\partial x_2} & \dots & \frac{\partial f_1}{\partial x_n} \smallskip \\
        f_2 & \dots & 0 & \frac{\partial f_2}{\partial x_1} & \frac{\partial f_2}{\partial x_2} & \dots & \frac{\partial f_2}{\partial x_n} \smallskip \\
        \vdots & \ddots & \vdots & \vdots & \vdots & \ddots & \vdots \smallskip \\
        0 & \dots & f_m & \frac{\partial f_m}{\partial x_1} & \frac{\partial f_m}{\partial x_2} & \dots & \frac{\partial f_m}{\partial x_n}  \smallskip \\
        s_2 & \cdots & s_m & \frac{\partial \ell_\cA}{\partial x_1} & \frac{\partial \ell_\cA}{\partial x_2} & \cdots & \frac{\partial \ell_\cA}{\partial x_n}
    \end{pmatrix}
\end{equation}
By (\ref{eq:idealwithdenominators}), these new matrix entries generate
the likelihood ideal $I_\cA$ in the localized ring $S_f$.
This implies that the
 maximal minors of $\Qi^{s,\partial}$ vanish on $\cL_\cA$; see also Proposition~\ref{prop:idealSummary}.  

For any maximal minor $D$, we
consider the Laplace expansion with
respect to the last~row:
\begin{equation}
    \label{eq:splitminor}
      D \,\,=\,\, \sum_{i\in I}  A_i(x) \cdot s_i \,+\, \sum_{j\in J} \theta_j(x) \cdot \frac{\partial \ell_\cA}{\partial x_j},
\end{equation}
Here $I \subset \{1,\dotsc, m\!-\!1\}$, $J \subset \{1,\dotsc,n\}$,
and the coefficients $A_i, \theta_j$ are polynomials in $ R$.  
We write $D(s,0)$ for the first sum in \eqref{eq:splitminor}, and we write
$D(0,\partial)$ for the derivation  $\sum_{j\in J} \theta_j \partial_j$.
The second sum in \eqref{eq:splitminor} is the evaluation of
$D(0,\partial) $ to the log-likelihood function $\ell_\cA$.

\begin{proposition}
    \label{prop:isoback} 
    If $\cA$ is SNC  then the matrix (\ref{eq:Qsdelta})
    and its minors (\ref{eq:splitminor}) 
realize the isomorphism
           between the module of tangent derivations and 
    the linear part of the likelihood~ideal:
$$
\varphi  \,\,:\,\, \Der(\cA) \,\simeq \,\left( I_\cA \right)_{(1,\bullet)},\,\,\,
D(0,\partial) \,\longleftrightarrow \,-D(s,0).
$$
\end{proposition}

\begin{proof}
The rational function $D$ in (\ref{eq:splitminor}) lies in the image
of the likelihood ideal $I_\cA$ in $S_f$.
\end{proof}

As an immediate consequence of Proposition~\ref{prop:isoback} we get  minimal generators of $\Der (\cA)$.

\begin{corollary}\label{cor:DerGen}
   The module $\Der(\cA)$ of an SNC arrangement $\cA$ with no hyperplanes is minimally generated by the $\binom{m+n-1}{n-2}$ circuit syzygies
   $D(0,\partial)$ and the Euler derivation $\sum_{j=1}^n x_j \partial_j$.
\end{corollary}

If $\cA$ contains hyperplanes,
then the number of circuit syzygies must
be corrected as in Theorem~\ref{thm:main2}. If all the $f_i$ are linear, then we are in the setting of a generic hyperplane arrangement. For such arrangements Rose--Terao~\cite{RoseTerao} and Yuzvinsky~\cite{Yuzvinsky} present a minimal free resolution of $\Der(\cA)$. Their main tool is a result of Lebelt \cite{lebelt1977freie} on constructing free resolutions of exterior powers of a module.  We now apply this in our non-linear situation.

Lebelt's idea is to build a resolution of an exterior power of a module from a known resolution of that module. The Euler derivation $\theta_E =\sum_{j=1}^n x_j \partial_j$ generates a free rank one summand of $\Der(A)$, and as in \cite{HalfreeCurves} we write $\Der(A) = \Der_0(A) \oplus R\theta_E$. The dual of  $\Der(\cA)$ is the module of logarithmic 1-forms with poles along~$\cA$, denoted $\Omega^1(\cA) = \Hom_R(\Der(\cA), R)$. The pruned version is $\ker(\Qi) \simeq \Der_0(\cA)$, and dually $\Omega_0^1(\cA) = \Hom_R(\Der_0(\cA), R)$. 

Lebelt's 
construction is applied to $\Der_0 (\cA) \simeq \Omega_0^{n-2}(\cA) = \bigwedge^{n-2}\Omega^1_0(\cA)$.  A short free resolution of $\Omega^1_0 (\cA)$ is given by observing that it equals the cokernel of $\Qi^T$:
\[
0 \to R^m \xrightarrow{\Qi^T} R^{m+n-1} \to \Omega^1_0(\cA) \to 0.
\]
Section~3.3.4 of \cite{MaxThesis} works out the details on constructing
Lebelt's resolution of $\Der_0(\cA)$.

Our Section~\ref{sec:det_ideals} offers an alternative route to a minimal free resolution of $\Der_0(\cA)$.
First note that
the Buchsbaum--Rim complex does not yield a resolution for $\Der(\cA)$. 
The matrix $Q$ is of size $m \times (m+n)$. In order for the Buchsbaum--Rim complex to be exact, the maximal minors of $Q$ would need to contain a regular sequence of length $n+1$.  This does not exist in $R=\CC[x_1,\ldots,x_n]$.  But it works for the pruned module $\Der_0(\cA)$.  By Proposition~\ref{circuitSyzgenerate}, the Buchsbaum--Rim complex is a resolution for~$\ker(\Qi)$.  A minimal free resolution of $\Der(\cA) \simeq \ker(Q)$ results after adding a direct summand to account for the Euler derivation.

We now relate our results to two favorable properties of arrangements: being tame and being gentle.  For SNC arrangements, the modules of logarithmic $p$-forms with poles along $\cA$ are  simply exterior powers.
Namely,
we have $\Omega^p(\cA) = \bigwedge^p \Omega^1(\cA)$ by \cite[Lemma~3.3.27]{MaxThesis}. 

Recall (e.g.~from \cite{cohen2011critical})
that an arrangement $\cA$ is {\em tame} if
\[ 
  \pd_R(\Omega^p(\cA)) \leq p \quad \text{for all }\,\, 0\leq p \leq n. 
  \]
In the hyperplane setting, tameness is a well-studied weakening of the freeness property, see \cite[Def.\ 2.2]{cohen2011critical}.  Using Lebelt's construction, short enough resolutions exist for each $\Omega^p(\cA)$.
\begin{theorem}[{\cite[Corollary~3.3.33]{MaxThesis}}]\label{cor:genericImpliesTame}
  Any SNC arrangement of hypersurfaces is tame.
\end{theorem}

Now we turn to gentleness \cite{ArrangementsAndLikelihood}.
 If $\begin{psmallmatrix}
    A\\
    B
\end{psmallmatrix}\in R^{(m+n) \times t}$ is a matrix that generates the kernel of~$Q$, where $A\in R^{m\times t}$ and $B\in R^{n\times t}$, then the \emph{likelihood module}  is $M(\cA) = \coker(A)$. In our SNC situation, if $d_i \ge 2$ for all $i$, this module has $m$ generators and $\binom{m+n-1}{n-2} + 1$ relations by Corollary~\ref{cor:DerGen}.  If there are linear forms, the modifications in Theorem~\ref{thm:main2} apply.  

The
\emph{pre-likelihood ideal} $I_0$ is derived from the matrix $A$ and generated by
the entries of $(s_1,\dotsc, s_m)\cdot A$.  Let $K = I_{m+1}(\Qi^s) + \langle \sum_i d_is_i \rangle$ as in Section~\ref{sec:det_ideals}. For any arrangement $\cA$,
\begin{equation}
    \label{eq:ideal_inclusions}
    K \,\subseteq \,I_0 \,\subseteq\, I_\cA.
\end{equation}

The arrangement $\cA$ is called \emph{gentle} if $I_0 = I_\cA$.  By \cite[Theorem~1.1]{ArrangementsAndLikelihood} the likelihood ideal $I_\cA$ is  the presentation ideal of the \emph{Rees algebra} $\R(M(\cA))$ of the likelihood module. The ideal $I_0$ in turn presents the symmetric algebra $\Sym(M(\cA)) \simeq S/I_0$.  They coincide if and only if the pre-likelihood ideal $I_0$ is prime~\cite[Proposition~2.9]{ArrangementsAndLikelihood}.  By \Cref{thm:main}, if $\cA$ is SNC, the inclusions in \eqref{eq:ideal_inclusions} are equalities.  In particular, we obtain the following statement.
\begin{corollary}\label{cor:genericImpliesGentle}
Every SNC arrangement is gentle.
\end{corollary}

This result was known in the linear case. A hyperplane arrangement
is SNC if and only if its matroid is uniform. This means
that all maximal minors of its defining matrix are non-zero (cf.~Example~\ref{ex:4linesP2}).
  It is known that linear SNC arrangements are tame.  And, tame linear arrangements are gentle, by an elaborate homological argument~\cite[Corollary~3.8]{cohen2011critical}.

Computing the likelihood ideal $I_\cA$ for a gentle arrangement is much easier than for
 arbitrary arrangements, where it requires a time-consuming saturation step~\cite[Remark~2.10]{ArrangementsAndLikelihood}.
 Theorems~\ref{thm:main} and~\ref{thm:main2} further simplify the situation.
Under the SNC hypothesis,  the ideal $I_\cA$ can be written down without any computation.
For a subsequent research project, one  might like
 to identify situations in which $K$ is close to $I_\cA$ or close to~$I_0$, or generally use the minors of $\Qi^s$ as an approximation of the likelihood correspondence.  
Studying the primary decompositions of such ideals for arrangements on the Euler discriminant 
will lead to interesting challenges.  One starting point would be when the hypersurfaces have isolated singular points, which occur away from the intersection of two or more hypersurfaces.

\smallskip

We remark in closing that  the literature in algebraic statistics
has mostly focused on implicit models, given by
a variety $\mathcal{V}$ of dimension $n-1$ in projective space $ \PP^{m-2}$, 
from which $m$ hyperplanes are removed \cite{HS14}. The analogue to our
SNC setting arises when $\mathcal{V}$ is a complete intersection.
It would be interesting to extend our results to 
the augmented Jacobian matrix in~\cite[Section 3]{HKS}.
Ideals and syzygies are now over the coordinate ring~$\CC[\mathcal{V}]$.

\bigskip \medskip

\noindent {\bf Acknowledgements}: 
Part of this project was done at the Max Planck Institute 
(MPI-MiS) in Leipzig.
TK and HS  thank MPI-MiS for its great working atmosphere. 
TK~is supported by the DFG (SPP 2458, 
539866293).
HS is supported by the NSF (DMS 2006410).
BS and MW are supported by the ERC (UNIVERSE+, 101118787). 
\ 
\begin{scriptsize} 
Views~and opinions 
expressed 
\vspace{-0.12cm}
are however those of the authors only and do not necessarily reflect those of the European Union or the European
Research Council Executive Agency. Neither the European Union nor the granting authority can be held responsible for them.
\end{scriptsize}

\let\OLDthebibliography\thebibliography
\renewcommand\thebibliography[1]{
  \OLDthebibliography{#1}
  \setlength{\parskip}{0pt}
  \setlength{\itemsep}{0pt plus 0.3ex}
}

\small

\bibliographystyle{plain}
\bibliography{bib.bib}
\normalsize

\bigskip \medskip

\noindent
\footnotesize {\bf Authors' addresses:}

\noindent Thomas Kahle, OvGU Magdeburg, Germany,
{\tt thomas.kahle@ovgu.de}

\noindent Hal Schenck, Auburn University, Alabama, US, {\tt hks0015@auburn.edu}

\noindent Bernd Sturmfels, MPI-MiS Leipzig, {\tt bernd@mis.mpg.de} and
UC~Berkeley, {\tt bernd@berkeley.edu}

\noindent Maximilian Wiesmann, Center for Systems Biology, Dresden, 
 {\tt wiesmann@pks.mpg.de}

\end{document}